\documentclass[
bibliography=totoc,fontsize=10pt]{scrartcl}

\usepackage[margin=1.1in]{geometry}  
\usepackage{graphicx}              
\usepackage{amssymb}
\usepackage{amsmath}               
\usepackage{amsfonts}              
\usepackage{amsthm}                
\usepackage{paralist}
\usepackage[usenames,dvipsnames]{color} 
\usepackage[british]{babel}
\usepackage{cite}
\usepackage[affil-it]{authblk}
\usepackage{xypic}
\usepackage{setspace}                
\usepackage{hyperref}

\usepackage[automark]{scrpage2}		
\clearscrheadfoot
\ohead{}
\lohead{\rightmark}
\rehead{\leftmark}

\ohead{\textbf{\pagemark}}

\setheadsepline{1pt} 

\addtokomafont{section}{\center}

\newtheorem{thm}{Theorem}[section]
\newtheorem{lem}[thm]{Lemma}
\newtheorem{prop}[thm]{Proposition}
\newtheorem{cor}[thm]{Corollary}

\newtheorem{defi}[thm]{Definition}
\newtheorem{rem}[thm]{Remark}

\newtheorem*{MainThm}{Main Theorem}
\newtheorem*{Ackn}{Acknowledgement}

\newenvironment{proofof}[1]{\renewcommand{\proofname}{Proof of #1}\proof}{\endproof}

\DeclareMathAlphabet\mathbb{U}{fplmbb}{m}{n}

\newcommand{\OO}{\mathbb{O}}
\newcommand{\HH}{\mathbb{H}}
\newcommand{\CC}{\mathbb{C}}
\newcommand{\RR}{\mathbb{R}}     

\newcommand{\ZZ}{\mathbb{Z}}     
\newcommand{\NN}{\mathbb{N}}

\newcommand{\defeq}{\mathrel{\mathop{\raisebox{1.1pt}{\scriptsize$:$}}}=}

\newcommand\opna{\operatorname}
\newcommand\mf{\mathfrak}
\newcommand\mc{\mathcal}
\newcommand\mass{\operatorname{mass}}
\newcommand\Lip{\operatorname{Lip}}

\addtokomafont{section}{\normalsize \rmfamily\scshape}
\addtokomafont{subsection}{\small \rmfamily\scshape}

\setkomafont{sectioning}{\normalcolor\rmfamily\scshape}

\begin{document}

\title{\vspace*{-10mm}\Large Higher divergence for nilpotent Lie groups}
\date{}
\author{\large Moritz Gruber \footnote{The author is supported by the German Research Foundation (DFG) grant GR 5203/1-1}}
\maketitle


\begin{abstract}\noindent\small
\textbf{Abstract. } The higher divergence of a metric space describes its isoperimetric behaviour at infinity. It is closely related to the higher-dimensional Dehn functions, but has more requirements to the fillings. We prove that these additional requirements do not have an essential impact for many nilpotent Lie groups. As a corollary, we obtain the higher divergence  of the Heisenberg groups in all dimensions.
\end{abstract}

%
%
\pagenumbering{arabic}

\section{Introduction}\label{S1}

Thales' theorem of intersecting lines is one the classical theorems in Euclidean geometry. In the words of a modern geometer, it says that geodesic rays in the Euclidean space diverge from each other linearly. As asymptotic property, the speed of the divergence of geodesics is an interesting object to study in many metric spaces. There is a way to describe it by the needed length to connect two points on the boundary of a ball of radius $r$ by a curve which do not intersect the interior of the ball. This can be considered as an isoperimetric inequality and hence generalised to higher dimensions: fill $k$-cycles outside a ball with $(k+1)$-chains also outside the ball.  
So far, this \emph{higher divergence} has mainly been studied for spaces with non-positive curvature (see \cite{BradyFarb}, \cite{Leuzinger2000}, \cite{Wenger06}, \cite{ABDDY}).

In this article we examine the higher divergence of nilpotent Lie groups with Riemannian metrics. These are metric spaces with all three types of sectional curvature: negative, zero and positive (see \cite{Wolf}). We find out that for a class of nilpotent Lie groups the behaviour in low dimensions are the same as for the filling functions. As a corollary, we obtain the higher divergence of the Heisenberg groups $H^n_\CC$ in all dimensions. As far as we know, this is the first infinite family of metric spaces, not quasi-isometric to non-positively curved spaces, for which the higher divergence is known in all dimensions.

\begin{defi}\quad\\
Let $n$ be a positive integer and $G$ be stratified nilpotent Lie group with a uniform lattice $\Gamma \subset G$ with $s_2(\Gamma) \subset \Gamma$. We call $G$ \emph{$n$-approximable} if there are biLipschitz triangulations $(\tau, f)$ of $G$ and $(\eta,g)$ of $G \times [0,1]$ and a $n$-horizontal, $\Gamma$-equivariant map $\phi: \tau \to G$ and a $n$-horizontal, $s_2(\Gamma)$-equivariant map $\psi: \eta \to G$ such that:
\begin{enumerate}[\quad 1)]
\item $\phi$ is in finite distance to $f$, and
\item there are isomorphisms $\iota_i: g^{-1}(G \times i) \to (\tau, s_{2^i} \circ f)$ such that $g \circ \iota_i = s_{2^i} \circ f$ and $\psi \circ \iota_i = s_{2^i} \circ \phi$ for $i \in \{0,1\}$.
\end{enumerate}
\end{defi}

A rich class of such $n$-approximable nilpotent Lie groups is formed by the \emph{Jet groups} $J^m(\RR^n)$ (compare \cite{Warhurst},\cite{Young1}) which include the Heisenberg group $H^n_\CC$ as $J^1(\RR^n)$.

\clearpage

\begin{MainThm}\quad\\
Let $G$ be a stratified nilpotent Lie group equipped with a left-invariant Riemannian metric and let $m \in \NN$ be such, that $G$ is $(k+1)$-approximable. Then:
$$\opna{Div}_G^k(r) \preccurlyeq r^{k+1}$$
and if there is a number $K \in \NN$ such that for all $\Delta \in \eta^{(k+2)}$ holds $\mass(s_{t}(\psi(\Delta))) \precsim t^{K+2}$, then
$$\opna{Div}_G^{k+1}(r) \preccurlyeq r^{K+2}.$$
\end{MainThm}
\quad\\

As the complex Heisenberg group $H^n_\CC$ is $k$-approximable for $k \le n$ (see \cite{Young1}), the Main Theorem combined with the results on the higher divergence from \cite{Gruber1} yields:
\begin{cor}\label{CorHeis}\quad\\
For the $(2n+1)$-dimensional complex Heisenberg group $H^n_\CC$ equipped with a left-invariant Riemannian metric 
$$\opna{Div}_{H^n_\CC}^k (r)\sim  	   \begin{cases} r^{k+1} \hspace{7mm} \text{ for } 1\le k <n \\ 
											r^{n+2} \hspace{7mm} \text{ for } k=n \\ 
											r^\frac{k(k+2)}{k+1} \hspace{4mm} \text{ for } n< k <2n 
									\end{cases} $$
\end{cor}
\begin{proof} Using the Main Theorem, the growth of the $k$-dimensional divergence follows by 
\cite[Corollary 4.14]{Young1} for $k<n$, by \cite[Theorem 8]{Young1} for $k=n$ and 
by \cite[Corrollary 5.4]{Gruber1} for $k>n$.
\end{proof}

Therefore, we now know the higher divergence in all dimensions for the complex Heisenberg groups. 

\begin{rem}\quad\\
Using the results of \emph{\cite{Gruber1}} and \emph{\cite{Gruber2}}, we get the corresponding behaviour for the divergence of the quaternionic and octonionic Heisenberg groups $H_\HH^n$ and $H_\OO^n$ in the dimensions $\le n$ and codimensions $<n$. Only for the octonionic Heisenberg groups $H_\OO^n$ we haven't the lower bound for the divergence in dimension $n$, because of the missing lower bound for the filling function in this dimension.
\end{rem}

\begin{Ackn}\quad\\
The author is pleased to thank Robert Young for the many enlightening discussions during the development of this article and to thank Enrico Leuzinger for initially drawing the authors attention to this topic.
\end{Ackn}
%

%
%

\subsection*{Illustration of our approach}

We illustrate the idea behind the proof of our Main Theorem in the case of the Euclidean plane $\RR^2$. To be more precise, we show how to construct an $\rho r$-avoidant $1$-chain with a prescribed $r$-avoidant boundary by using an approximation technique. This technique will generalise to nilpotent Lie groups and higher dimensions (see section \ref{FaY}).\\

Let $\tau$ be the grid of squares of side length $1$ with vertices in $\ZZ^2$ and let $\tau_i=2^i\tau$ for $i \in \NN_0$ be the scaled grid of squares with side length $2^i$, and let $r>1$. \\

\begin{minipage}{0.3\textwidth}
\includegraphics[scale=0.3]{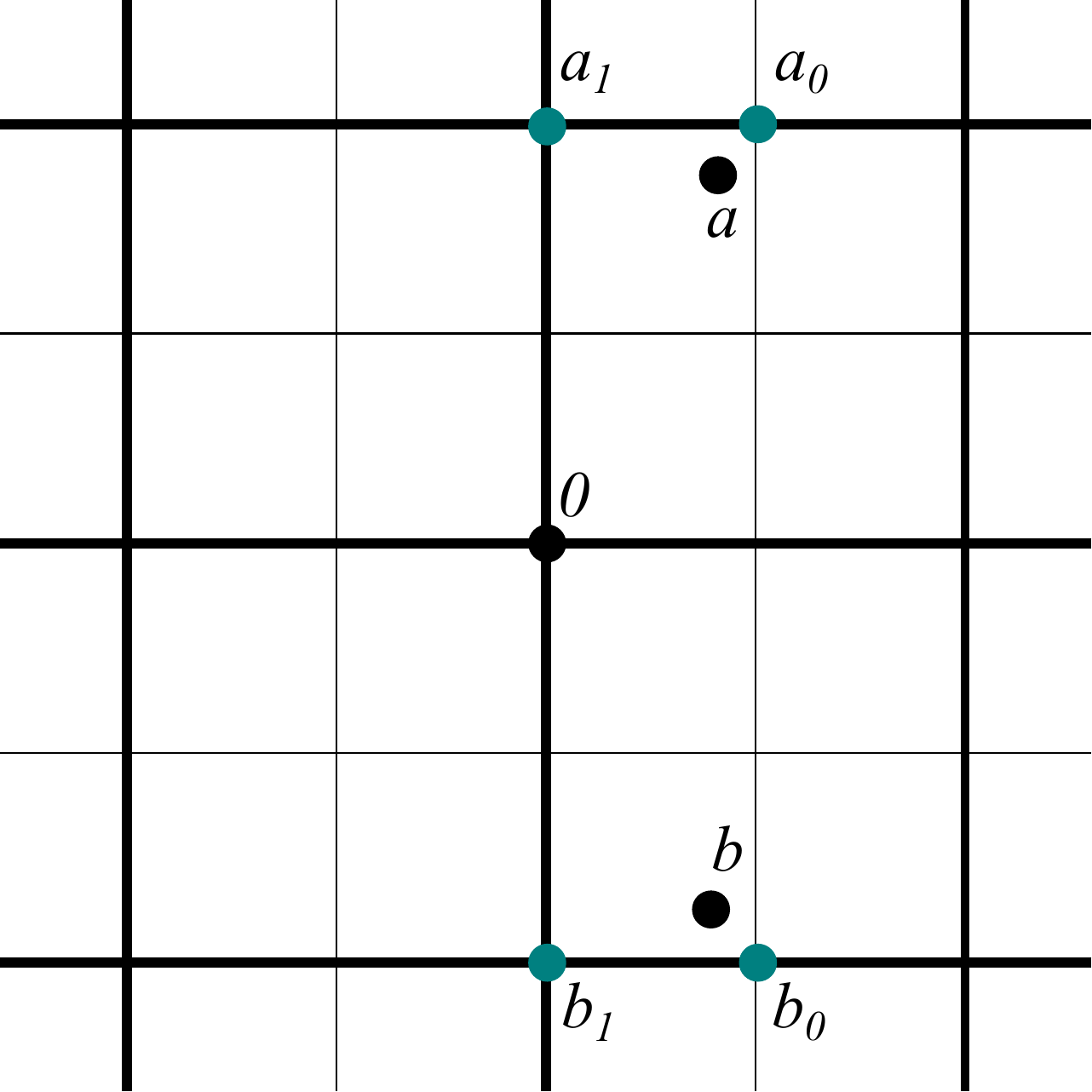}
\quad\\
\end{minipage}
\hfill
\begin{minipage}{0.65\textwidth}
Given a $0$-cycle $B=a-b$ with $d(B, 0)\ge r$, one approximates it by $B_0=a_0-b_0$ in the $0$-skeleton of the grid $\tau$.

Then one approximates $B$ in the grid $\tau_1$ by $B_1=a_1-b_1$. And further on by $0$-cycles $B_i=a_i-b_i$ in the $0$-skeleton of the grids $\tau_i$.

After finitely many steps, latest for $i_o$ with  $2^{i_o-2} \le d(a,b) < 2^{i_o-1}$, the points $a$ and $b$ will be approximated by the same point and therefore the cycle $B_{i_o}$ is trivial.\\
\vspace{3mm}
\end{minipage}

\begin{minipage}{0.3\textwidth}
\includegraphics[scale=0.3]{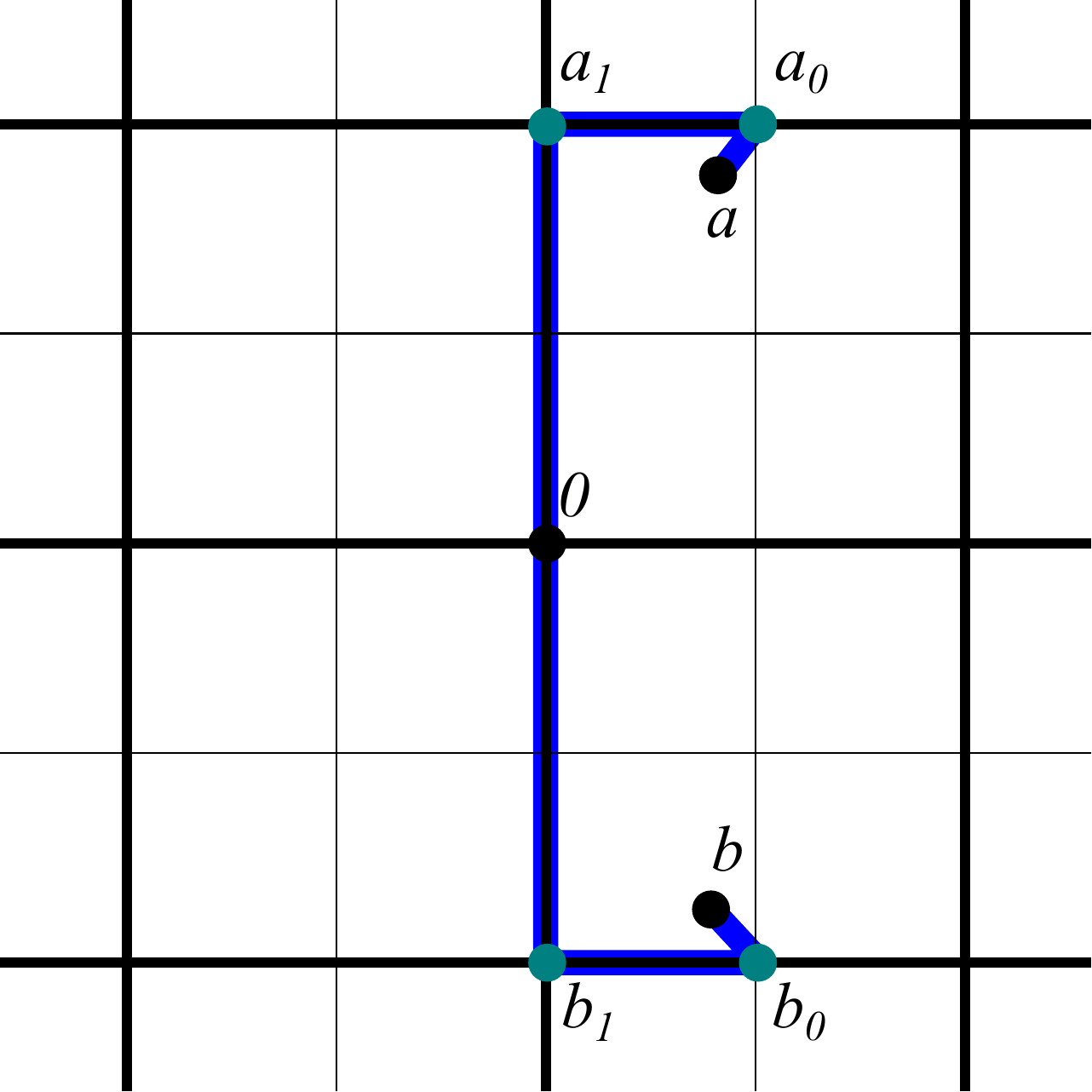}
\vspace{15mm}
\end{minipage}
\hfill
\begin{minipage}{0.65\textwidth}
Having all the above approximations of the $0$-cycle $B$, one can start to construct the filling. First, one connects the original points $a$ and $b$ by straight lines with their approximations $a_0$ and $b_0$. These lines are not longer than half the diameter of a square in $\tau$, i.e. at most $\frac{\sqrt{2}}{2}$. This is independent of the cycle.

Then one connects two consecutive approximations $B_i$ and $B_{i+1}$ by edges in the $1$-skeleton of $\tau_i$. One never needs more than two such edges to connect $a_i$ with $a_{i+1}$ and $b_i$ with $b_{i+1}$, and hence ends up with a $1$-chain $F$ connecting $a$ and $b$ and with 
$$\mass (F)\le \frac{\sqrt{2}}{2} + \sum _{i=0}^{i_o-1} 2\cdot 2^i \le \frac{\sqrt{2}}{2} + 2^{i_o+1} \le \frac{\sqrt{2}}{2} + 8 \cdot d(a,b) \sim d(a,b).$$
\end{minipage}
\quad\\\quad\\

The crux for our purpose is, as seen above, that even when the cycle to fill is far away from the base point $0$, the constructed filling may use edges containing the base point. To avoid this, we replace the base point when ever it is used in the approximation by one of its neighbours $x_1,x_2,x_3,x_4$ in the $0$-skeleton of the respective $\tau_i$, say $x_1$, and replace each edge connecting $0$ with $x_j, \ j\in \{1,2,3,4\}$ by a $1$-chain, not passing through $0$, connecting $x_1$ with the same $x_j$.\\

\begin{figure}[h]
\centering
\includegraphics[width=130mm]{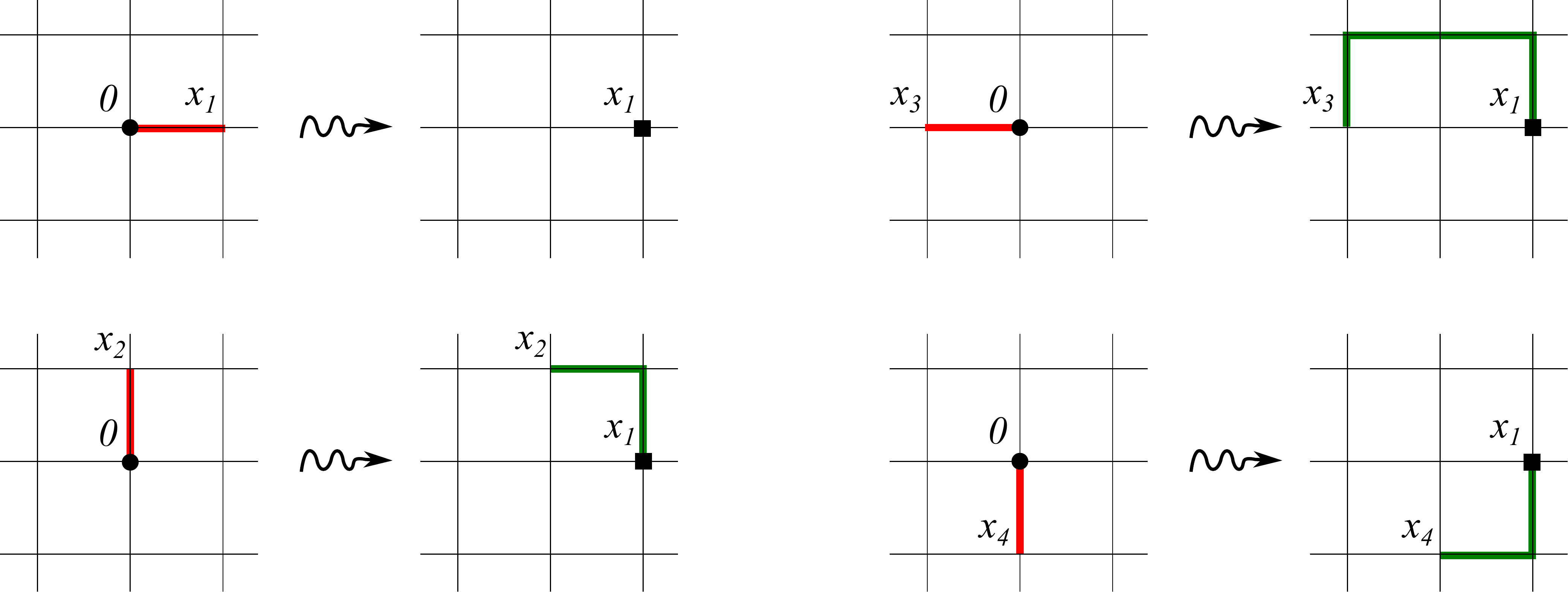}\vspace{3mm}
\end{figure}

We have to replace only finitely many edges, here $4$, and therefore get a maximal number $M$ of edges used in a replacement chain, here $4$, too.\\\quad\\

\begin{minipage}{0.3\textwidth}
\includegraphics[scale=0.3]{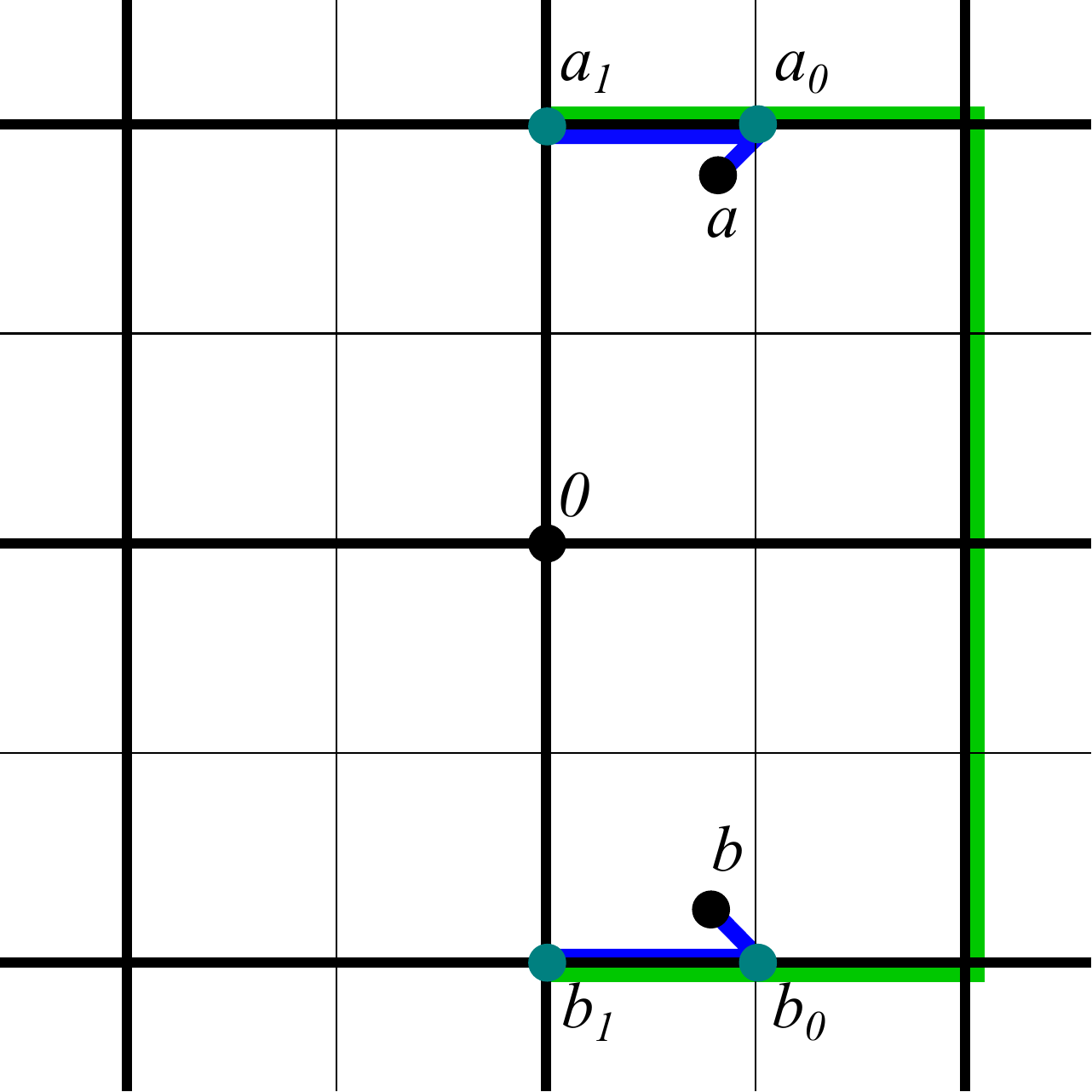}
\vspace{7mm}
\end{minipage}
\hfill
\begin{minipage}{0.65\textwidth}
If we now replace all the edges used in $F$ which contain $0$, we obtain a new $1$-chain $\widetilde F$ connecting $a$ and $b$. As we never substitute an edge by more than $M=4$ edges, all of the same length as the original edge, we get
$$\mass(\widetilde F) \le M \cdot \mass(F) \sim d(a,b).$$
Furthermore, the used simplices of $\tau_i$ will always stay in the $(2^i\sqrt{2})$-neighbourhood of $B$ and will never be nearer to $0$ than $2^i$. Hence 
$$d(\widetilde F, 0) \ge \min \big\{ \max\{d(B, 0)-2^i\sqrt{2},\ 2^i\} \mid i \in \NN_0 \big\} \ge \frac{1}{1+\sqrt{2}} r.$$ 

\end{minipage}


\section{Preliminaries}\label{S2}

\subsection{Higher Divergence}

Let $X$ be a Riemannian manifold or a simplicial complex with base point $x_o \in X$. For $r>0$ we call a Lipschitz chain $b$ in $X$ \emph{$r$-avoidant} if the image of $b$ has trivial intersection with the $r$-ball around $x_o$. The \emph{higher divergence} gives a measure of the difficulty to fill $r$-avoidant Lipschitz cycles by almost $r$-avoidant Lipschitz chains. We now give a brief definition of it, for more details see \cite{BradyFarb}, \cite{ABDDY} and \cite{Doktorarbeit}. 

\begin{defi}\quad\\
For $\rho \in (0,1]$ and $\alpha >0$ we set
$$\opna{div}^m_{X,\rho,\alpha}(r) \defeq \sup\big\{\inf\{\mass(b) \mid \partial b=a, \ b\ \rho r\text{-avoidant}\} \mid a\ r\text{-avoidant $m$-cycle, } \mass(a)\le \alpha r^m  \big\}.$$
The \emph{$m$-dimensional divergence} of $X$ is the $2$-parameter family
$$\opna{Div}^m_X \defeq\{\opna{div}^m_{X,\rho,\alpha}\}_{\rho,\alpha}$$
with $\rho \in (0,1]$ and $\alpha >0$.
\end{defi}
Usually one wants to see the higher divergence as a quasi-isometry invariant. For this purpose one considers its equivalence class by the relation defined in the following.

\begin{defi}\quad\\
Let $F=\{f_{s,t}\}$ and $H=\{h_{s,t}\}$ be two $2$-parameter families indexed over $s \in (0,1]$ and $t>0$ with increasing functions $f_{s,t}, h_{s,t}: \RR^+ \to \RR^+\cup \{\infty\}$. For fixed $m \in \NN$ we write $F \preccurlyeq_m H$ if there are thresholds $s_o$ and $t_o$, and constants $L,M \ge 1$ such that for all $s\le s_o$ and all $t\ge t_o$ there is a constant $A=A(s,t) \ge 1$ with
$$f_{s,t}(x) \le Ah_{Ls,Mt}(Ax+A) + O(x^m) \quad \forall x \in \RR^+.$$
If $F \preccurlyeq_m H$ and $H \preccurlyeq_m F$, we write $F \sim_m H$. This defines an equivalence relation.
\end{defi}

In the case of the $m$-dimensional divergence we will use the relations $\preccurlyeq_m$ and $\sim_m$. For brevity we will omit the index and only write $\preccurlyeq$ and $\sim$. For the same reason we denote the $2$-parameter family $F=\{f_{s,t}\}$ consisting of the same function $f$ for all indices, i.e. $f_{s,t}=f$ for all $s,t$, shortly by this function $f$. Further we will use for functions $f,h : \RR^+ \to \RR^+\cup \{\infty\}$ the notation $f\precsim h$ when there is a constant $C>0$ such that $f(x) \le Ch(x)+Cx+C$ for all $x\in \RR_+$.

\begin{rem}\label{Remr}\quad\\
The relation $\preccurlyeq$ (and consequently $\sim$) only captures the asymptotic behaviour of the functions for $r \to \infty$. 
Let $h:\RR_+ \to\RR_+\cup \{\infty\}$ be an increasing function. If  $r_o\ge1$ and $\opna{div}^m_{X,\rho,\alpha}(r_o)\le h(r_o)$ then $\opna{div}^m_{X,\rho,\alpha}(r)\le r_o\cdot h(r_o r+ r_o)  $ for all $r \le r_o$, because both sides are increasing. So we need to examine the relation $\preccurlyeq$ (and consequently $\sim$)  only for $r$ larger than an arbitrary constant $r_o=r_o(\rho,\alpha)\ge 1$.
\end{rem}


\subsection{Stratified nilpotent Lie groups}

We call a simply connected, $d$-step nilpotent Lie group $G$ with Lie algebra $\mf g$ \emph{stratified} if there are subspaces $V_i \subset \mf g$ with
$$\mf g = V_1 \oplus V_2 \oplus ... \oplus V_d \quad \text{and} \quad [V_j, V_1]=V_{j+1}$$
where $V_k=\{0\}$ for all $k>d$. For every stratified nilpotent Lie group $G$ there is a family $\{\hat s_t\}_{t>0}$ of Lie algebra automorphisms $\hat s_t: \mf g \to \mf g, (v_1,v_2,...,v_d) \mapsto (tv_1,t^2v_2,...,t^dv_d)$. These induce (uniquely determined) Lie group automorphisms $s_t : G \to G$ with differential $\hat s_t$ at the neutral element $1 \in G$, the so called \emph{scaling automorphisms} or \emph{dilations}. A special role in the geometry of stratified nilpotent Lie groups plays the \emph{horizontal distribution}
$$\mc H \defeq \bigcup_{g \in G} \opna{d}\!L_g V_1 \subset TG.$$
Let $G$ be equipped with a left-invariant Riemannian metric. A Lipschitz map $f:M \to G$ from a Riemannian manifold $M$ to $G$ is \emph{horizontal} if its image is almost everywhere tangent to $\mc H$. By this, one can define the \emph{Carnot-Carath\'eodory metric} $d_c(\cdot,\cdot)$ as the length metric of horizontal curves with the Riemannian length. This is a left-invariant metric on $G$ with 
$$d_c(s_t(x),s_t(y))=t \cdot d_c(x,y) \quad \forall x,y \in G\ \forall t>0.$$
If we denote the Riemannian distance by $d(\cdot,\cdot)$, one has $d(x,y)\le d_c(x,y)$ for all $x,y \in G$. Further, both metrics induce the same topology on $G$. This leads to the following useful estimate.

\begin{lem}\label{LemSkal}\quad\\
Let $G$ be a stratified nilpotent Lie group equipped with a left-invariant Riemannian metric and let $d(\cdot,\cdot)$ be the induced length metric and let $r>0$. Then there is a constant $C_r>0$ such that:
$$d(s_t(x),s_t(y))\le t \cdot C_r \cdot r \quad \forall t >0\ \forall x,y \in G \text{ with } d(x,y)\le r.$$
In particular, for all $x,y \in G$  with $d(x,y)=r$ this implies
$$d(s_t(x),s_t(y))\le t \cdot C_r \cdot d(x,y) \quad \forall t >0.$$
\end{lem}
\begin{proof}
Let $d_c(\cdot,\cdot)$ be the corresponding Carnot-Carath\'eodory metric. 
We have $x\in B_r(y) \defeq \{z \in G \mid d(z,y)\le r\}$. 
As $B_r(y)$ is compact and $z \mapsto d_c(z,y)$ is continuous, we have 
$$C_r\defeq \max \left\{\frac{d_c(z,y)}{r} \mid z \in B_r(y)\right\} < \infty.$$
This implies
$$d(s_t(x),s_t(y))\le d_c(s_t(x),s_t(y))=t \cdot d_c(x,y)=t \cdot \frac{r}{r} \cdot d_c(x,y)=t \cdot \frac{d_c(x,y)}{r} \cdot r\le t \cdot C_r \cdot r$$
for all $y \in G$ and all $x \in B_r(y)$.
\end{proof}


\section{Filling by approximation in nilpotent groups }\label{FaY}

Young developed in \cite{Young1} a technique for nilpotent Lie groups to construct fillings by approximation. In this section we give a reminder of this construction. In the second half of the section we will introduce a modification of it which allows us to avoid simplices near the base point.\\

A main ingredient for Young's construction is the following approximation theorem. 

\begin{thm}\textup{(Federer-Fleming Deformation Theorem, \cite{FF60})}\\
Let $X$ be a polyhedral complex with finitely many isometry types of cells. Then there is a constant $C_X>0$ depending only on $X$, such that for every Lipschitz $k$-chain $b$, whose boundary $\partial b$ is a polyhedral $(k-1)$-chain, there is a polyhedral $k$-chain $P_X(b)$ and a Lipschitz $(k+1)$-chain $Q_X(b)$ with:
\begin{enumerate}[\quad 1)]
	\item $\mass(P_X(b))\le C_X \cdot \mass(b)$,				\vspace{-2mm}
	\item $\mass(Q_X(b))\le C_X \cdot \mass(b)$,				\vspace{-2mm}
	\item $\partial Q_X(b)=b-P_X(b)$.				
\end{enumerate}
Furthermore, $P_X(b)$ and $Q_X(b)$ are contained in the smallest subcomplex of $X$ which contains $b$.
\end{thm}

In the following, we start with a recapitulation of the notation and the filling technique from \cite{Young1}. For a Lipschitz map $\varphi: X \to Y$ between to metric spaces $X$ and $Y$, we denote by $\varphi_\#$ the map that sends the Lipschitz chain $a=\sum_i z_i \alpha_i$ in $X$ to the Lipschitz chain $\varphi_\#(a)=\sum_i z_i (\varphi\circ \alpha_i)$ in $Y$. 

Let $G$ always denote an $n$-dimensional, stratified nilpotent Lie group equipped with a left-invariant Riemannian metric and let $\Gamma \le G$ be a (uniform) lattice in $G$ with $s_2(\Gamma) \subset \Gamma$.
Let $(\tau, f:\tau \to G)$ be a biLipschitz triangulation of $G$ and $\phi: \tau \to G$ be a $(k+1)$-horizontal, $\Gamma$-equivariant Lipschitz map in bounded distance $c_\phi> 0$ to $f$. Hence, there is a Lipschitz homotopy $h:G \times [0,1] \to G$ from $\opna{Id}_G$ to $\phi \circ f^{-1}$.

If $a$ is a Lipschitz $k$-cycle in $G$, then $f^{-1}_\#(a)$ is a Lipschitz $k$-cycle in $\tau$ and one can define in $G$ the Lipschitz $k$-cycle
\begin{equation*}\label{defP}
	P_{\phi(\tau)}(a)\defeq \phi_\#\left(P_\tau(f^{-1}_\#(a))\right).
\end{equation*}
\begin{lem}\label{LemQ}\emph{(\cite[Lemma 3.1]{Young1})}\\
There is a constant $c_Q$, only depending on $\tau$, $\phi$ and $k$, such that if\\
	$Q_{\phi(\tau)}(a)\defeq h(a \times [0,1]) + \phi_\#\left(Q_\tau(f^{-1}_\#(a))\right)$
then
\begin{enumerate}[\quad 1)]
\item $\partial Q_{\phi(\tau)}(a)=a - P_{\phi(\tau)}(a)$, \vspace{-2mm}
\item $\mass(Q_{\phi(\tau)}(a)) \le c_Q \cdot \mass(a)$.
\end{enumerate}
\end{lem}
By defining $P_i(a)\defeq s_{2^i}\left(P_{\phi(\tau)}(s_{2^{-i}}(a))\right)$
one obtains a sequence of rougher and rougher approximations of $a$ and can prove that there is a constant $c_P$, only depending on $\tau$, $\phi$ and $k$, such that $\mass(P_i(a)) \le c_P \cdot \mass(a)$ for all $i \in \NN_0$ (\cite[Lemma 3.2]{Young1}).

In the following crucial step one constructs chains interpolating between two consecutive approximations $P_i$ and $P_{i+1}$. For this, denote for $i\in\{0,1\}$ by $(\tau_i,f_i)$ the triangulation $(\tau, s_{2^i}\circ f)$ of $G$ and let $(\eta, g:\eta \to G\times[0,1])$ be a biLipschitz triangulation of $G\times[0,1]$ such that  the subcomplexes $g^{-1}(G \times \{i\})$, $i \in \{0,1\}$, are isomorphic to $\tau_i$ by isomorphisms $\iota_i$ with $g\circ \iota_i=f_i$. Further let $\psi:\eta \to G$ be a $(k+1)$-horizontal, $s_2(\Gamma)$-equivariant Lipschitz map and define $\phi_i: \tau_i \to G$ by $\phi_i\defeq \psi \circ \iota_i$.
By \cite[Proposition 4.5]{Young1} such a triangulation $(\eta,g)$ always exists, when $(\tau,f)$ is chosen $\Gamma$-equivariant for a  lattice $\Gamma$ of $G$ with $s_2(\Gamma) \subset \Gamma$.

\begin{lem}\emph{(\cite[Lemma 3.3]{Young1})}\\
There is a constant $c_R$, only depending on $\eta$, $\psi$ and $k$, such that if
\begin{equation}\label{defX}
	X(a)\defeq Q_{\tau_1}((f_1^{-1})_\#(a))+g^{-1}_\#(a\times[0,1]) - Q_{\tau_0}((f_0^{-1})_\#(a))
\end{equation}
and $R_{\psi(\eta)}(a)\defeq \psi_\# \left(P_\eta(X(a))\right)$ then
\begin{enumerate}[\quad 1)]
\item $\partial R_{\psi(\eta)}(a)= P_{\phi_1(\tau_1)}(a)-P_{\phi_0(\tau_0)}(a)$, \vspace{-2mm}
\item $\mass (R_{\psi(\eta)}(a)) \le c_R \cdot \mass(a)$.
\end{enumerate}
\end{lem}

By scaling the chain $R_{\psi(\eta)}$ in the same way as the cycle $P_{\phi(\tau)}$ one gets	$R_i(a)\defeq s_{2^i}\left(R_{\psi(\eta)}(s_{2^{-i}}(a))\right)$ and the following holds for all $i\in \NN_0$ (\cite[Lemma 3.4]{Young1}):
\begin{enumerate}
\item[\textit{1)}] $\partial R_i(a)= P_{i+1}(a)-P_{i}(a)$ \vspace{-2mm}
\item[\textit{2)}] $\mass (R_i(a)) \le c_R \cdot 2^i \cdot \mass(a)$
\end{enumerate}
Using all these constructions, Young shows (\cite[Theorem 5]{Young1}) that for any Lipschitz $k$-cycle $a$ and  $i_o \in \NN$ such that 
$2^{(i_o-1)k} \le c_\Delta \mass(a) \le 2^{i_ok}$ with $c_\Delta \defeq c_\tau \opna{Lip}(f^{-1})^k \frac{1}{\mass(\Delta^{(k)})}$, the Lipschitz $(k+1)$-chain 
\begin{equation*}\label{defb}
	b\defeq - \left(Q_{\phi(\tau)}(a)+ \sum_{i=0}^{i_o-1} R_i(a) \right)
\end{equation*}
is a filling of $a$ with $\mass(b) \precsim \mass(a)^{\frac{k+1}{k}}$.

Furthermore, the same construction yields upper bounds on the filling functions if one replaces the horizontality condition on the map $\psi$ by a bound on the scaling behaviour of $\psi$-images of simplices (compare \cite[Section 5]{Young1}). For this let $f:\RR^+\to \RR^+$ be a function so that for every $(d+1)$-simplex $\Delta \in \eta^{(d+1)}$ one has $\mass(s_t(\psi(\Delta))) \le f(t)$. For every Lipschitz $d$-cycle $a$ in $G$ we have $\mass(P_\eta(X(a))) \le c_\tau \left( 2 c_\tau \opna{Lip}(g^{-1})^d + \opna{Lip}(g^{-1})^{d+1} \right) \mass(a)$ and therefore $P_\eta(X(a))$ consists of not more than $\frac{c_\tau \left( 2 c_\tau \opna{Lip}(g^{-1})^d + \opna{Lip}(g^{-1})^{d+1} \right)}{\mass(\Delta^{(d+1)})} \mass(a)$ many $(d+1)$-simplices. For this reason and with $c_X \defeq \frac{c_\tau \left( 2 c_\tau \opna{Lip}(g^{-1})^d + \opna{Lip}(g^{-1})^{d+1} \right)}{\mass(\Delta^{(d+1)})}$ one has
\begin{equation*}
\mass(R_i(a)) = \mass(s_{2^i}\big(\psi_\#(P_\eta(X(s_{2^{-i}}(a))))\big)) \le f(2^i)c_X \mass(s_{2^{-i}}(a))\le 2^{-di}c_Xf(2^i) \mass(a).
\end{equation*}
If there is a $D\in \NN$ such that one has the above situation with a function $f(t) \sim t^D$, this leads to a filling $b\defeq -Q_{\phi(\tau)}(a)- \sum_{i=0}^{i_o-1} R_i(a)$ with $\mass(b) \precsim \mass(a)^\frac{D}{d}$.\\


We would like to use these constructions to fill $r$-avoidant cycles by $\rho r$-avoidant chains for some $\rho \in (0,1]$. Unfortunately, one can't exclude the occurrence of chains $R_i$ containing the base point $1\in G$ or getting arbitrarily near to it.

To fill $r$-avoidant cycles, we modify this construction to get a control on the distance of the chains $R_i$ to the neutral element $1\in G$ of the group $G$. To do this, we replace simplices too near to $1$ by chains in larger distance.\\

Let $\varepsilon >0$ and let $\{w_1,...,w_m\}$ be the set of all $(n+1)$-simplices of $\eta$ which are simplices  $\Delta^{(n+1)}$ or share a vertex with a simplex $\Delta^{(n+1)}$ with 
$$\psi(\Delta^{(n+1)})\cap B_\varepsilon(1) \ne \emptyset.$$
Let $\{v_1,...,v_{m_\tau}\}$ be the set of all simplices of $\tau$ which have, as simplices of $\tau_0$ or $\tau_1$, non-empty intersection with one of the simplices of $\{w_1,...,w_m\}$. Both of these sets are finite as $\psi$ is $s_2(\Gamma)$-equivariant. As $\tau$ is homeomorphic to $\RR^n$, we find a finite subcomplex $V_{\tau,\varepsilon}$ of it which contains all the simplices $v_i$ and is homeomorphic to an $n$-ball and therefore has an $(n-2)$-connected boundary.

As the first step for our modification, we choose a vertex $x_o \in (\partial V_{\tau,\varepsilon})^{(0)}$ and define the map
$$\pi_0: V_{\tau,\varepsilon}^{(0)} \to (\partial V_{\tau,\varepsilon})^{(0)},\ x \mapsto \pi_0(x)$$
with
\begin{enumerate}[\quad (1)]
\item $\pi_0(x)=x \quad \,\ \forall x \in (\partial V_{\tau,\varepsilon})^{(0)}$,
\item $\pi_0(x)=x_o \quad \forall x \in V_{\tau,\varepsilon}^{(0)} \setminus (\partial V_{\tau,\varepsilon})^{(0)}$.
\end{enumerate}
We denote by $C^{(l)}(T),\ T\subset \tau$, the integral $l$-chains in the subcomplex $T$. Then we inductively define for $1\le l \le n-1$ the maps
$$\pi_l: V_{\tau,\varepsilon}^{(l)} \to C^{(l)}(\partial V_{\tau,\varepsilon}),\ y \mapsto \pi_l(y)$$
with
\begin{enumerate}[\quad ($\pi_l$1)]
\item $\pi_l(y)=y \quad \,\ \forall y \in (\partial V_{\tau,\varepsilon})^{(l)}$,
\item if $\partial y=\sum_{j=0}^l (-1)^j y_j$, so is $\partial \pi_l(y)=\sum_{j=0}^l (-1)^j \pi_{l-1}(y_j)$.
\end{enumerate}
We can construct such maps by filling the boundaries of the mapped simplices by $l$-chains in $\partial V_{\tau,\varepsilon}$. This is possible as $\partial V_{\tau,\varepsilon}$ is $(n-2)$-connected and $l\le n-1$.

Each of the skeletons $V_{\tau,\varepsilon}^{(l)}$ is finite because $V_{\tau,\varepsilon}$ is finite. For this reason, the maximum
$$M_l \defeq \max \left\{  \frac{\mass(\pi_l(y))}{\mass(\Delta^{(l)})} \mid y \in V_{\tau,\varepsilon}^{(l)} \right\} < \infty$$
exists for all $l \in \{0,...,n-1\}$.

Using property $(\pi_l1)$, we can extend the maps $\pi_l$ to the whole $l$-skeleton of $\tau$ by the identity. By property $(\pi_l2)$, we can extend this map to arbitrary $l$-chains:
$$\widetilde{\pi_l}: C^{(l)}(\tau) \to C^{(l)}((\tau \setminus V_{\tau,\varepsilon}) \cup \partial V_{\tau,\varepsilon}),\ b=\sum_{j} \beta_j y_j \mapsto \widetilde{\pi_l}(b) \defeq \sum_j \beta_j \pi_l(y_j).$$
As no $l$-simplex is mapped to a $l$-chain of mass more than $M_l \cdot \mass(\Delta^{(l)})$, we get
$$\mass(\widetilde{\pi_l}(b)) \le M_l \cdot \mass(b) \quad \forall b \in C^{(l)}(\tau) \ \forall l \in \{0,...,n-1\}$$
We defined the maps $\widetilde{\pi_l}$ on the $l$-chains of the simplicial complex $\tau$. As $\tau_0$ and $\tau_1$ are isomorphic to $\tau$, we can consider the maps $\widetilde{\pi_l}$ as maps on $l$-chains of these subcomplexes of $\eta$, too.

As next step, we define similar maps on the skeletons of the simplicial complex $\eta$. For this let $W_{\eta,\varepsilon}$ be a finite subcomplex of $\eta$ which contains all the simplices $w_j, \ j \in \{1,...,m\}$, is homeomorphic to an $(n+1)$-ball and such that $W_{\eta,\varepsilon} \cap \tau_i = V_{\tau_i,\varepsilon},\ i \in \{0,1\}$, where $V_{\tau_i,\varepsilon}$ is the copy of $V_{\tau,\varepsilon}$ in $\tau_i$.
Then the subcomplexes $V_{\tau_i,\varepsilon},\ i \in \{0,1\}$, are contained in the boundary of $W_{\eta,\varepsilon}$. We denote $\widetilde W \defeq \big(\partial W_{\eta,\varepsilon} \setminus (V_{\tau_0,\varepsilon} \cup V_{\tau_1,\varepsilon})\big)\cup (\partial V_{\tau_0,\varepsilon} \cup \partial V_{\tau_1,\varepsilon})$ and choose a vertex $x_1 \in \widetilde W^{(0)}$. Then we define
$$p_0: W_{\eta,\varepsilon}^{(0)} \to \widetilde W^{(0)},\ x \mapsto p_0(x)$$
with 
\begin{enumerate}[\quad (1)]
\item $p_0(x)=\pi_0(x) 	\quad 		\forall x \in V_{\tau_0,\varepsilon}^{(0)} \cup V_{\tau_1,\varepsilon}^{(0)}$,
\item $p_0(x)=x \hspace{9.75mm} 	\forall x \in (\partial W_{\eta,\varepsilon})^{(0)} \setminus (V_{\tau_0,\varepsilon}^{(0)} \cup V_{\tau_1,\varepsilon}^{(0)})$,
\item $p_0(x)=x_1 \hspace{8.25mm}	\forall x \in W_{\eta,\varepsilon}	{(0)} \setminus (\partial W_{\eta,\varepsilon})^{(0)}$.
\end{enumerate}
We denote by $C^{(l)}(E),\ E\subset \eta$, the integral $l$-chains in the subcomplex $E$. Then we inductively define for $1\le l \le n-1$ the maps
$$p_l: W_{\eta,\varepsilon}^{(l)} \to C^{(l)}(\widetilde W), y \mapsto p_l(y)$$
with
\begin{enumerate}[\quad ($p_l$1)]
\item $p_l(y)=\pi_l(y) \quad 			\forall y \in V_{\tau_0,\varepsilon}^{(l)} \cup V_{\tau_1,\varepsilon}^{(l)}$,
\item $p_l(y)=y \hspace{9.25mm}	\forall y \in (\partial W_{\eta,\varepsilon})^{(l)} \setminus (V_{\tau_0,\varepsilon}^{(l)} \cup V_{\tau_1,\varepsilon}^{(l)})$,
\item if $\partial y = \sum_{j=0}^l (-1)^jy_j$, so is $\partial p_l(y)=\sum_{j=0}^l (-1)^jp_{l-1}(y_j)$.
\end{enumerate}
We are able to construct such maps by filling the boundaries of the mapped simplices by $l$-chains in $\widetilde W$. This is possible as $\widetilde W$ is homeomorphic to an $n$-dimensional cylinder and therefore $(n-2)$-connected and because $\pi_l$ satisfies property $(p_l3)$ by property $(\pi_l2)$.

As every skeleton $W_{\eta,\varepsilon}^{(l)}$ is finite, the maximum
$$M'_l \defeq \max \left\{\frac{\mass(p_l(y))}{\mass(\Delta^{(k)})} \mid y \in W_{\eta,\varepsilon}^{(l)}\right\} < \infty$$
exists. 

By property $(p_l2)$ and by property $(p_l1)$, combined with property $(\pi_l1)$, we can extend the maps $p_l$ to the whole $l$-skeleton of $\eta$ by the identity. Using further property $(p_l3)$, we can extend them to maps
$$\widetilde{p_l}: C^{(l)}(\eta) \to C^{(l)}((\eta \setminus W_{\eta,\varepsilon})\cup \widetilde W), \ b=\sum_j \beta_j y_j \mapsto \widetilde{p_l}(b) \defeq \sum_j \beta_j p_l(y_j)$$
of $l$-chains.

Again, no $l$-simplex is mapped to an $l$-chain of mass more than $M'_l \cdot \mass(\Delta^{(l)})$ and so we obtain
$$\mass(\widetilde{p_l}(b)) \le M'_l \cdot \mass(b) \quad \forall b \in C^{(l)}(\eta)\ \forall l \in\{0,...,n-1\}.$$

We use the maps $\widetilde{p_l}$ to adapt the chains $P_i$ and $R_i$ (compare above) for our purpose to fill avoidant cycles with avoidant chains. Let $r>0$ and let $a$ be an $r$-avoidant Lipschitz $k$-cycle in $G$. We define
\begin{equation*}
\widetilde{P_{\phi(\tau)}}(a)\defeq \phi_\#\left( \widetilde{p_k} (P_\tau (f_\#^{-1}(a)))\right)
\end{equation*}
and
\begin{equation}\label{Ptildei}
\widetilde{P_i}(a) \defeq s_{2^i}(\widetilde{P_{\phi(\tau)}}(s_{2^{-i}}(a)))
\end{equation}
as well as
\begin{equation}\label{Rtilde}
\widetilde{R_{\psi(\eta)}}(a) \defeq \psi_\#\left(\widetilde{p_{k+1}}(P_\eta(X(a))) \right) 
\end{equation}
and
\begin{equation}\label{Rtildei}
\widetilde{R_i}(a) \defeq s_{2^i}(\widetilde{R_{\psi(\eta)}}(s_{2^{-i}}(a))).
\end{equation}

\begin{lem}\label{LemRtilde}\quad\\
For the above defined chains holds
$$\partial \widetilde{R_i}(a)=\widetilde{P_{i+1}}(a) -\widetilde{P_{i}}(a) .$$
Further there is a constant $c_{\widetilde{R}}$, only depending on the triangulation $(\eta,g)$, the map $\psi$ and the map $\widetilde{p_{k+1}}$, such that for all $k$-cycles $a$ as above 
$$\mass(\widetilde{R_i}(a))\le 2^i c_{\widetilde{R}}\mass(a).$$
\end{lem}
\begin{proof}
By equation (\ref{defX}) we have $\partial X(a)= P_{\tau_1}(f_\#^{-1}(a)) - P_{\tau_0}(f_\#^{-1}(a))$. Therefore
\begin{align*}
\partial \widetilde{R_{\psi(\eta)}}(a)=& \psi_\#(\partial \widetilde{p_{k+1}}(P_\eta(X(a))) \stackrel{(p_{k+1}3)}{=}\psi_\#\big(\widetilde{p_{k}}(\partial P_\eta(X(a)))\big)=\psi_\#\big(\widetilde{p_{k}}(P_{\tau_1}(f_\#^{-1}(a)) - P_{\tau_0}(f_\#^{-1}(a)))\big)\\
=& \widetilde{P_1}(a)- \widetilde{P_0}(a)
\end{align*}
So we obtain
\begin{align*}
\partial \widetilde{R_i}(a)\stackrel{(\ref{Rtildei})}{=}&s_{2^i}(\widetilde{\partial R_{\psi(\eta)}}(s_{2^{-i}}(a)))=s_{2^i}( \widetilde{P_1}(s_{2^{-i}}(a))- \widetilde{P_0}(s_{2^{-i}}(a))) \\
\stackrel{(\ref{Ptildei})}{=}&s_{2^i}( s_2(\widetilde{P_0}(s_{2^{-1}}(s_{2^{-i}}(a)))))- s_{2^i}(\widetilde{P_0}(s_{2^{-i}}(a)))\stackrel{(\ref{Ptildei})}{=} \widetilde{P_{i+1}}(a)-\widetilde{P_{i}}(a)
\end{align*}
For the estimate of the mass we use $\mass(X(a)) \le \left(2c_\eta \cdot \opna{Lip}(g^{-1})^k +  \opna{Lip}(g^{-1})^{k+1} \right) \mass(a)$ from the proof of \cite[Lemma 3.3]{Young1}, where $c_\eta$ denotes the constant in the Federer-Fleming Deformation Theorem for the complex $\eta$. With this we get
\begin{align*}
\mass(\widetilde{R_{\psi(\eta)}}(a)) &\le \opna{Lip}(\psi)^{k+1}\mass(\widetilde{p_{k+1}}(P_\eta(X(a)))\\
&\le \opna{Lip}(\psi)^{k+1} M'_{k+1} \mass(P_\eta(X(a)))\\
&\le \opna{Lip}(\psi)^{k+1} M'_{k+1}c_\eta \mass(X(a)) \\
&\le \opna{Lip}(\psi)^{k+1} M'_{k+1}c_\eta\left(2c_\eta \cdot \opna{Lip}(g^{-1})^k +  \opna{Lip}(g^{-1})^{k+1} \right) \mass(a)\ .
\end{align*}
As the chain $\widetilde{R_{\psi(\eta)}}(a)$ is horizontal, it follows with 
$$c_{\widetilde{R}}\defeq\opna{Lip}(\psi)^{k+1} M'_{k+1}c_\eta\left(2c_\eta \cdot \opna{Lip}(g^{-1})^k +  \opna{Lip}(g^{-1})^{k+1} \right)$$
that
$$\mass(\widetilde{R_i}(a)) = 2^{i(k+1)}\mass(\widetilde{R_{\psi(\eta)}}(s_{2^{-i}}(a))) \le 2^{i(k+1)} c_{\widetilde{R}} 2^{-ik} \mass(a) = 2^ic_{\widetilde{R}} \mass(a). $$
\end{proof}

We close our modification by constructing a filling of a cycle $a$ consisting of the old chain $Q_{\phi(\tau)}(a)$ and the new chains $\widetilde{R_i}(a)$. Let $r>\opna{Lip}(\psi)(\opna{diam}(W_{\eta, \varepsilon})+ \opna{diam}(\Delta^{(n+1)}))>0$, so that for every $r$-avoidant $k$-cycle $a$ in $G$ the preimage $f_\#^{-1}(a)$ does not intersect $V_{\tau,\varepsilon}$. Let $a$ be an $r$-avoidant $k$-cycle. Therefore $P_\tau(f_\#^{-1}(a))$ uses no simplices of $V_{\tau,\varepsilon} \setminus \partial V_{\tau,\varepsilon}$ as it is contained in the smallest subcomplex of $\tau$ that contains $f_\#^{-1}(a)$. In particular, we have $\widetilde{P_0(a)}=P_0(a)$. 

\begin{prop}\label{Propb}\quad\\
Let $a$ be as above, $c_\Delta \defeq c_\tau \opna{Lip}(f^{-1})^k \frac{1}{\mass(\Delta^{(k)})}$ and $i_o \in \NN$ such that 
$2^{(i_o-1)k} \le c_\Delta \mass(a) \le 2^{i_ok}$. Then the $(k+1)$-chain
$$\widetilde{b} \defeq -\left(Q_{\phi(\tau)}(a) + \sum_{i=0}^{i_o-1} \widetilde{R_{i}}(a)\right)$$
is a filling of $a$ with $\mass(\widetilde{b}) \precsim \mass(a)^\frac{k+1}{k}$.
\end{prop}
\begin{proof}
By the proof of \cite[Theorem 3]{Young1} we know $P_\tau(s_{2^{-i_o}}(a))=0$ and therefore \\
$\widetilde{P_{i_o}}(a) = s_{2^{i_0}}(\phi_\#(\widetilde{p_{k}}(P_\tau(s_{2^{-i_o}}(a)))))=0$. As $\widetilde{P_0}(a)=P_0(a)$, we have
\begin{align*}
\partial \widetilde{b}=& - \left(\partial Q_{\phi(\tau)}(a) + \sum_{i=0}^{i_o-1} \partial R_{i}(a)\right) =-\left((P_0(a)-a) + \sum_{i=0}^{i_o-1} (\widetilde{P_{i+1}}(a)-\widetilde{P_i}(a))\right)\\
=&-\left((P_0(a)-a) + (\widetilde{P_{i_o}}(a)-\widetilde{P_0}(a))\right) = a.
\end{align*}
With Lemma \ref{LemQ} and Lemma \ref{LemRtilde} we get
\begin{align*}
\mass(\widetilde{b}) \le& c_Q\mass(a) + \sum_{i=0}^{i_0-1}(2^ic_{\widetilde{R}}\mass(a)) \le  c_Q\mass(a) + 2^{i_o}c_{\widetilde{R}}\mass(a)\\
\stackrel{\text{def } i_o}{\le}& (c_Q + 2 c_{\widetilde{R}}(c_\Delta \mass(a))^\frac{1}{k}) \mass(a).
\end{align*}
For $\mass(a) \ge 1$ this implies $\mass(\widetilde{b}) \le (c_Q + 2 c_{\widetilde{R}}(c_\Delta)^\frac{1}{k} ) \mass(a)^\frac{k+1}{k} \precsim \mass(a)^\frac{k+1}{k}$.
\end{proof}


Again, if one replaces the horizontality condition on the map $\psi$ by a bound on the scaling behaviour of $\psi$-images of simplices, one can still bound the mass of the filling $\widetilde b$. 
\begin{lem}\label{Lemf}\quad\\
Let $f:\RR^+\to \RR^+$ be a function so that for every $(d+1)$-simplex $\Delta \in \eta^{(d+1)}$ one has $\mass(s_t(\psi(\Delta))) \le f(t)$. Then there is a constant $\widetilde{c_X}$ such that for every Lipschitz $d$-cycle $a$ in $G$
\begin{equation}
\mass(\widetilde{R_i}(a)) \le 2^{-di}\widetilde{c_X}f(2^i) \mass(a).
\end{equation}
\end{lem}
\begin{proof}
We have 
$$\mass(\widetilde{p_{d+1}}(P_\eta(X(a)))) \le M'_{d+1}c_\tau \left( 2 c_\tau \opna{Lip}(g^{-1})^d + \opna{Lip}(g^{-1})^{d+1} \right) \mass(a)$$
and therefore $\widetilde{p_{d+1}}(P_\eta(X(a)))$ consists of not more than 
$$ \frac{1}{\mass(\Delta^{(d+1)})}M'_{d+1}c_\tau \left( 2 c_\tau \opna{Lip}(g^{-1})^d + \opna{Lip}(g^{-1})^{d+1} \right) \mass(a)$$
many $(d+1)$-simplices. For this reason and with $\widetilde{c_X} \defeq \frac{M'_{d+1}c_\tau \left( 2 c_\tau \opna{Lip}(g^{-1})^d + \opna{Lip}(g^{-1})^{d+1} \right)}{\mass(\Delta^{(d+1)})}$ we get
\begin{align*}
\mass(\widetilde{R_i}(a)) &= \mass(s_{2^i}\big(\psi_\#(\widetilde{p_{d+1}}(P_\eta(X(s_{2^{-i}}(a)))))\big))\\
&\le f(2^i)\widetilde{c_X} \mass(s_{2^{-i}}(a))\\
&\le 2^{-di}\widetilde{c_X}f(2^i) \mass(a).
\end{align*}
\end{proof}

If one can use a polynomial function $f$ this leads to the following.

\begin{prop}\label{PropFillD}\quad\\
Let $D \in \NN$, $D>d\ge 1$, and $C >0$ be so that for every $(d+1)$-simplex $\Delta \in \eta^{(d+1)}$ one has $\mass(s_t(\psi(\Delta))) \le Ct^D+Ct+C$. Then for every $r$-avoidant Lipschitz $d$-cycle $a$ the Lipschitz $(d+1)$-chain
$$\widetilde{b} \defeq -\left(Q_{\phi(\tau)}(a) + \sum_{i=0}^{i_o-1} \widetilde{R_{i}}(a)\right)$$
is a filling of $a$ with $\mass(\widetilde{b}) \precsim \mass(a)^\frac{D}{d}$.
\end{prop}
\begin{proof}
The fact that $\widetilde b$ is a filling of $a$ we have already seen in Proposition \ref{Propb}. 
As we now can use Lemma \ref{Lemf} with $f(t)=Ct^D+Ct+C$ we obtain for the mass
\begin{align*}
\mass(\widetilde{b}) &\le c_Q\mass(a) + \sum_{i=0}^{i_0-1}(2^{-di}(C(2^i)^D+C2^i+C)\widetilde{c_X}\mass(a)) \\
&\le c_Q\mass(a) + \sum_{i=0}^{i_0-1}C((2^{(D-d)i}+2^{(1-d)i}+2^{-di})\widetilde{c_X}\mass(a))\\
&\le c_Q\mass(a) + \sum_{i=0}^{i_0-1}C((2^{(D-d)i}+2\cdot2^{(1-d)i})\widetilde{c_X}\mass(a))\\
&\le c_Q\mass(a) + (2^{(D-d)i_o}+2^{(1-d)i_o})C\widetilde{c_X}\mass(a)\\
&\le (c_Q + 2 C\widetilde{c_X}\left((c_\Delta \mass(a))^\frac{D-d}{d}+(c_\Delta \mass(a))^\frac{1-d}{d}\right)) \mass(a).
\end{align*} 
For $\mass(a) \ge 1$ this yields $\mass(\widetilde b)\le (c_Q + 2 C \widetilde{c_X}((c_\Delta )^\frac{D-d}{d}+ c_\Delta)) \mass(a)^\frac{D}{d} \precsim \mass(a)^\frac{D}{d}$.
\end{proof}


\section{The proof of the main result}\label{S3}

In this section we use all the notations as in section \ref{FaY}. To prove our Main Theorem we will show that in early approximation steps the chains $\widetilde{R_i}(a)$ stay near the cycle $a$ and that in late approximation steps they stay far away from the base point.

By construction we have that the  $(k+1)$-chains $\widetilde{p_{k+1}}(P_\eta(X(a)))$ only use simplices $\Delta$ of $\eta$ with $\psi(\Delta) \cap B_\varepsilon(1)= \emptyset$. 

Set $D_\eta \defeq \opna{diam}( \Delta^{(n+1)})$ and let $r>\opna{Lip}(\psi)(\opna{diam}(W_{\eta, \varepsilon})+ D_\eta)>0$ as before.

\begin{lem}\label{LemDk}\quad\\
Let $a$ be an $r$-avoidant Lipschitz $k$-cycle in $G$. Then holds 
$$d(\widetilde{R_i}(a),1) \ge 2^i \cdot\varepsilon.$$
\end{lem}
\begin{proof}
By the defining equation (\ref{Rtilde}) we have $\widetilde{R_{\psi(\eta)}}(a) \defeq \psi_\#\left(\widetilde{p_{k+1}}(P_\eta(X(a))) \right)$ and therefore \\$d(\widetilde{R_{\psi(\eta)}}(a), 1) > \varepsilon$.
By equation (\ref{Rtildei}) we have $\widetilde{R_i}(a)=s_{2^i}(\widetilde{R_{\psi(\eta)}}(s_{2^{-i}}(a)))$ and so
$d(\widetilde{R_i}(a),1)\ge  2^i  d(\widetilde{R_{\psi(\eta)}}(s_{2^{-i}}(a)),1) \ge 2^i \cdot\varepsilon$.
\end{proof}

\begin{lem}\label{LemDtilde}\quad\\
Let $a$ be an $r$-avoidant Lipschitz $k$-cycle in $G$. Then the chain $\widetilde{R_i}(a)$ is contained in the $(2^i\cdot C_{\widetilde D} \cdot \widetilde D)$-neighbourhood of $a$, where  $\widetilde D$ only depends on the triangulation $(\eta, g)$ and the maps $\phi$ and $\psi$ and the choice of the subcomplex $W_{\eta,\varepsilon}$ and $C_{\widetilde D}$ is the constant from Lemma \ref{LemSkal}.
\end{lem}
\begin{proof}
We prove that $\widetilde{R_{\psi(\eta)}}(a)$ is contained in the $\widetilde D$-neighbourhood of $a$ for $\widetilde D \defeq c_\phi + \Lip(\psi)\cdot(\Lip(g^{-1}) + D_\eta+ \opna{diam}(W_{\eta,\varepsilon}))$. Then the claim follows by Lemma \ref{LemSkal} and equation (\ref{Rtildei}).

Remember the definition $\widetilde{R_{\psi(\eta)}}(a)=\psi_\#(\widetilde{p_{k+1}}(P_\eta(X(a))))$ with 
$$X(a)= Q_{\tau_1}((f_1^{-1})_\#(a))+g^{-1}_\#(a\times[0,1]) - Q_{\tau_0}((f_0^{-1})_\#(a)).$$
As $f_0$ and $f_1$ are restrictions of $g$ and as $Q_{\tau_i}((f_i^{-1})_\#(a))$ is contained in the smallest subcomplex of $\tau_i$ that contains $(f_i^{-1})_\#(a)$, we have that $P_\eta(X(a))$ is contained in the smallest subcomplex of $\eta$ that contains $g^{-1}_\#(a\times[0,1])$. As further $g^{-1}_\#(a\times[0,1])$ lies in the $\Lip(g^{-1})$-neighbourhood of $(f_0^{-1})_\#(a)$, this subcomplex is contained in the $(\Lip(g^{-1})+D_\eta)$-neighbourhood of $(f_0^{-1})_\#(a)$. Further the map $\widetilde{p_{k+1}}$ moves chains at most by $\opna{diam}(W_{\eta,\varepsilon})$.
 Therefore $\widetilde{R_{\psi(\eta)}}(a)$ is contained in the $\big(\Lip(\psi)\cdot(\Lip(g^{-1})+D_\eta+\opna{diam}(W_{\eta,\varepsilon}))\big)$-neighbourhood of $\psi_\#((f_0^{-1})_\#(a))=\phi_\#((f^{-1})_\#(a))$. Finally, $\phi$ and $f$ are in distance at most $c_\phi$ and so $\phi_\#((f^{-1})_\#(a))$ is in the $c_\phi$-neighbourhood of $a$.
\end{proof}

\begin{lem}\label{LemDQ}\quad\\
Let $a$ be a Lipschitz $k$-cycle in $G$. Then there is a constant $D_Q>0$, only depending on the triangulations $(\tau,f)$, $(\eta,g)$, the map $\phi$ and the homotopy $h$, such that $Q_{\phi(\tau)}(a)$ is contained in the $D_Q$-neighbourhood of $a$.
\end{lem}
\begin{proof}
Remember the definition $Q_{\phi(\tau)}(a)=h_\#(a\times [0,1]) + \phi_\#(Q_\tau((f^{-1})_ \#(a)))$. As $Q_\tau((f^{-1})_ \#(a))$ is contained in the smallest subcomplex of $\tau$ which contains $(f^{-1})_ \#(a)$, and as $\tau$ is a subcomplex of $\eta$, it is contained in the $D_\eta$-neighbourhood of $(f^{-1})_ \#(a)$. So $\phi_\#(Q_\tau((f^{-1})_\#(a)))$ is contained in the $(\Lip(\phi) \cdot D_\eta)$-neighbourhood of $\phi_\#((f^{-1})_ \#(a))$. As $\phi$ and $f$ are in distance at most $c_\phi$, $\phi_\#((f^{-1})_\#(a))$ is in the $c_\phi$-neighbourhood of $a$. We get that $Q_{\phi(\tau)}(a)$ is contained in the $D_Q$-neighbourhood of $a$ for $D_Q\defeq \Lip(h)+c_\phi+ \Lip(\phi)\cdot D_\eta$.
\end{proof}

Now we are prepared for the proof of our Main Theorem.

\begin{proofof}{the Main Theorem}\quad\\
Let $r>0$, $\alpha >0$. By Remark \ref{Remr} we can assume $r> 10 \cdot \max\{\varepsilon,D_Q,C_{\widetilde D}\cdot\widetilde D\}$. Let further 
$$\rho \le \min\left\{\frac{9}{10}\ ,\  \frac{\varepsilon}{\varepsilon+C_{\widetilde D}\cdot \widetilde D}\right\}.$$
We will show that for every $r$-avoidant Lipschitz $k$-cycle $a$ in $G$ with $\mass(a) \le \alpha \cdot r^k$
the filling $\widetilde b=-Q_{\phi(\tau)}(a)- \sum_{i=0}^{i_o-1} \widetilde{R_i}(a)$  is $\rho r$-avoidant.
For this we treat the $Q$-part and the $R$-part separately. 

By the lower bound on $r$ we have in particular $r > 10 \cdot D_Q$ and therefore $D_Q < \frac{1}{10} r$. As $a$ is $r$-avoidant, $Q_{\phi(\tau)}(a)$ is $\frac{9}{10} r$-avoidant by Lemma \ref{LemDQ} and therefore $\rho r$-avoidant.

For the $R$-part we use Lemma \ref{LemDk} and Lemma \ref{LemDtilde}. By these we have for every $i\in \{0,...,i_o-1\}$ that $\widetilde{R_i}(a)$ is simultaneously $(2^i \cdot \varepsilon)$-avoidant and $(r-2^i\cdot C_{\widetilde D}\cdot \widetilde D)$-avoidant. For every single $i$ we can use the bigger of these two numbers, and in the whole we get:
\begin{align*}
d(1,\sum_{i=0}^{i_o-1} \widetilde{R_i}(a)) &\ge \min\left\{ \max\{2^i \cdot \varepsilon\ ,\ r-2^i\cdot C_{\widetilde D}\cdot \widetilde D\} \mid i\in \{0,...,i_o-1\} \right\}\\
&\ge \min\left\{ \max\{2^i \cdot \varepsilon\ ,\ r-2^i\cdot C_{\widetilde D}\cdot \widetilde D\} \mid i\in\NN_0 \right\}\\
&\ge \min\left\{ \max\{2^i \cdot \varepsilon\ ,\ r-2^i\cdot C_{\widetilde D}\cdot \widetilde D\} \mid i\in [0,\infty) \right\}
\end{align*}
We set $M_1(i)\defeq 2^i \cdot \varepsilon$ and $M_2(i)\defeq r-2^i\cdot C_{\widetilde D}\cdot \widetilde D$ for $i \ge 0$. The function $M_1$ is strictly increasing in $i$ and the function $M_2$ is strictly decreasing in $i$. So we find the minimal maximum of $\{M_1(i),M_2(i)\}$ in the point where $M_1$ and $M_2$ coincide. This is for
$$M_1(i)=M_2(i)\ \Leftrightarrow \ 2^i \cdot \varepsilon=r-2^i\cdot C_{\widetilde D}\cdot \widetilde D\ \Leftrightarrow\ i=\log_2\left(\frac{r}{\varepsilon+C_{\widetilde D} \cdot \widetilde D}\right)$$
and hence the minimal maximum is
$$M_1(\log_2\left(\frac{r}{\varepsilon+C_{\widetilde D} \cdot \widetilde D}\right))= 2^{\log_2\left(\frac{r}{\varepsilon+C_{\widetilde D} \cdot \widetilde D}\right)} \cdot \varepsilon=\frac{r}{\varepsilon+C_{\widetilde D} \cdot \widetilde D} \cdot \varepsilon=\frac{\varepsilon}{\varepsilon+C_{\widetilde D} \cdot \widetilde D} \cdot r \ge \rho r\ .$$
Therefore $\sum_{i=0}^{i_o-1} \widetilde{R_i}(a)$ is $\rho r$-avoidant. Altogether, the filling $b$ is $\rho r$-avoidant and satisfies
$$\mass(\widetilde b) \precsim \mass(a)^\frac{k+1}{k} \precsim r^{k+1}.$$

If we start with na $r$-avoidant $(k+1)$-cycle $a'$ with $\mass(a')\le \alpha \cdot r^{k+1}$, the same computations show, that the filling $\widetilde{b'}=-Q_{\phi(\tau)}(a')- \sum_{i=0}^{i'_o-1} R_i(a')$ is $\rho r$-avoidant. 
As we have $\mass(s_{t}(\psi(\Delta))) \precsim t^{K+2}$ for all $\Delta \in \eta^{(k+2)}$ the filling $\widetilde{b'}$ satisfies by Proposition \ref{PropFillD}
$$\mass(\widetilde{b'}) \precsim \mass(a')^\frac{K+2}{k+1} \precsim r^{K+2}.$$
This completes the proof.
\end{proofof}



\bibliography{bib}
\bibliographystyle{plain}

\quad\\
\textsc{Courant Institute of Mathematical Sciences, New York University, New York, USA}\\
\hspace*{4mm}\textit{E-mail address:} moritz.gruber@cims.nyu.edu

\end{document}